\documentclass[oneside]{amsart}
\usepackage{amssymb}
\usepackage{amsxtra}
\usepackage{amsmath}
\usepackage{amsthm}
\usepackage[all]{xy}

\usepackage{hyperref}

\usepackage[english]{babel}

\DeclareMathOperator{\SO}{SO}

\DeclareMathOperator{\HH}{H}

\DeclareMathOperator{\Gm}{{\mathbf G}_m}

\DeclareMathOperator{\Aut}{Aut\,}

\DeclareMathOperator{\id}{id}

\DeclareMathOperator{\Str}{Str}
\DeclareMathOperator{\PG}{PG}
\DeclareMathOperator{\chr}{char}

\DeclareMathOperator{\et}{\text{\it \'et}}

\newtheorem{lem}{Lemma}
\newtheorem*{thm*}{Theorem}
\newtheorem{thm}{Theorem}

\newtheorem*{cor*}{Corollary}

\title{Grothendieck---Serre conjecture for groups of type $F_4$ with trivial $f_3$ invariant}
\thanks{V. Petrov is supported by Max-Plank-Institut f\"ur Mathematik, RFBR~08-01-00756
and RFBR~09-01-90304; V. Petrov and
A. Stavrova are supported by RFBR~09-01-00878}
\author{V. Petrov}
\author{A. Stavrova}
\begin{document}

\begin{abstract}
Assume that $R$ is a semi-local regular ring containing an infinite perfect field.
Let $K$ be the field of fractions of $R$.
Let $H$ be a simple algebraic group of type $F_4$ over $R$ such that $H_K$ is the automorphism group
of a 27-dimensional Jordan algebra which is a first Tits construction. If $\chr K\neq 2$
this means precisely that the $f_3$ invariant of $H_K$ is trivial.
We prove that the kernel of the map
$$
\HH^1_{\et}(R,H)\to \HH^1_{\et}(K,H)
$$
induced by the inclusion of $R$ into $K$ is trivial.

This result is a particular case of the Grothendieck---Serre conjecture on rationally trivial torsors.
It continues the recent series of
papers~\cite{PaSV}, \cite{Pa-newpur}, \cite{PaPS} and complements the result of Chernousov \cite{Ch}
on the Grothendieck---Serre conjecture for groups of type $F_4$ with trivial $g_3$ invariant.
\end{abstract}

\maketitle

\section{Introduction}

In the present paper we address the Grothendieck---Serre conjecture~\cite[p. 31, Remarque]{Serre},~\cite[Remarque 1.11]{Gr}
on the rationally trivial torsors of reductive algebraic groups. This conjecture states that for any
reductive group scheme $G$ over a regular ring $R$, any $G$-torsor that is trivial over the field of fractions $K$
of $R$ is itself trivial; in other words, the natural map
$$
H^1_{\text{\'et}}(R,G)\to H^1_{\text{\'et}}(K,G)$$
has trivial kernel.
It has been settled in a variety of particular cases,
and we refer to~\cite{Pa-newpur} for a detailed overview.
The most recent result belongs to V. Chernousov~\cite{Ch} who has proved that
the Grothendieck---Serre conjecture holds for an arbitrary simple group $H$ of type $F_4$ over a local regular
ring $R$ containing the field of rational numbers, given that $H_K$ has a trivial $g_3$ invariant.
We prove that the Grothendieck---Serre conjecture holds for another natural class of groups $H$ of type $F_4$,
those for which $H_K$ has trivial $f_3$ invariant. In fact, since our approach is characteristic-free,
we establish the following slightly more general result.

\begin{thm}\label{th:main}
Let $R$ be a semi-local regular ring containing an infinite perfect field.
Let $K$ be the field of
fractions of $R$. Let $J$ be a 27-dimensional exceptional Jordan algebra over $R$ such that $J_K$ is a first Tits
construction. 
Then the map
$$
\HH^1_{\et}(R,\,\Aut(J)) \to \HH^1_{\et}(K,\,\Aut(J))
$$
induced by the inclusion of $R$ into $K$ has trivial kernel.
\end{thm}

\begin{cor*}
Let $R$ be a semi-local regular ring containing an infinite perfect field $k$ such that $\chr k\neq 2$.
Let $K$ be the field of
fractions of $R$. Let $H$ be a simple group scheme of type $F_4$ over $R$ such that
$H_K$ has trivial $f_3$ invarint.
Then
the map
$$
\HH^1_{\et}(R,\,H) \to \HH^1_{\et}(K,\,H)
$$
induced by the inclusion of $R$ into $K$ has trivial kernel.
\end{cor*}

\section{Isotopes of Jordan algebras}

In the first two sections $R$ is an arbitrary commutative ring.

A \emph{(unital quadratic) Jordan algebra} is a projective $R$-module $J$
together with an element $1\in J$ and an operation
\begin{align*}
&J\times J\to J\\
&(x,\,y)\mapsto U_xy,
\end{align*}
which is quadratic in $x$ and linear in $y$ and satisfies the following axioms:
\begin{itemize}
\item $U_1=\id_J$;
\item $\{x,\,y,\,U_xz\}=U_x\{y,\,x,\,z\}$;
\item $U_{U_xy}=U_xU_yU_x$,
\end{itemize}
where $\{x,\,y,\,z\}=U_{x+z}y-U_xy-U_zy$ stands for the linearization of $U$.
It is well-known that the split simple group scheme of type $F_4$ can be realized
as the automorphism group scheme of the split $27$-dimensional exceptional
Jordan algebra $J_0$. This implies that any other group scheme of type $F_4$
is the automorphism group scheme of a twisted form of $J_0$.

Let $v$ be an \emph{invertible} element of $J$ (that is, $U_v$ is invertible).
An \emph{isotope} $J^{(v)}$ of $J$ is a new Jordan algebra whose underlying
module is $J$, while the identity and $U$-operator are given by the formulas
\begin{align*}
&1^{(v)}=v^{-1};\\
&U^{(v)}_x=U_xU_v.
\end{align*}
An \emph{isotopy} between two Jordan algebras $J$ and $J'$ is an isomorphism
$g\colon J\to J'^{(v)}$; it follows that $v=g(1)^{-1}$. We are particularly
interested in \emph{autotopies} of $J$; one can see that $g$ is an autotopy
if and only if
$$
U_{g(x)}=gU_xg^{-1}U_{g(1)}
$$
for all $x\in J$. In particular, transformations of the form $U_x$ are
autotopies. The group scheme of all autotopies is called the
\emph{structure group} of $J$ and is denoted by $\Str(J)$. Obviously it contains
$\Gm$ acting on $J$ by scalar transformations.

It is convenient to describe isotopies as isomorphisms of some algebraic
structures. This was done by O.~Loos who introduced the notion of a
\emph{Jordan pair}. We will not need the precise definition, see
\cite{Lo75} for details. It turns out that every Jordan algebra $J$ defines
a Jordan pair $(J,\,J)$, and the isotopies between $J$ and $J'$ bijectively
correspond to the isomorphims of $(J,\,J)$ and $(J',\,J')$
(\cite[Proposition~1.8]{Lo75}). In particular, the structure group
$\Str(J)$ is isomorphic to $\Aut((J,\,J))$.
We use this presentation of $\Str(J)$ to show that,
if $J$ is a $27$-dimensional exceptional Jordan algebra, $\Str(J)$ can be seen as a Levi
subgroup of a parabolic subgroup of type $P_7$ (with the enumeration of roots as in \cite{Bu}) in an adjoint group
of type $E_7$.
See also Garibaldi~\cite{Ga}.


\begin{lem}\label{lem:Str}
Let $J$ be a 27-dimensional exceptional Jordan algebra over a commutative ring $R$.
There exists
an adjoint simple group $G$ of type $E_7$ over $R$ such that $\Str(J)$ is isomorphic to a Levi subgroup
$L$ of a maximal parabolic subgroup $P$ of type $P_7$ in $G$.
\end{lem}
\begin{proof}
By~\cite[Theorem~4.6 and Lemma~4.11]{Lo78} for any Jordan algebra $J$
the group $\Aut((J,\,J))$
is isomorphic to a Levi subgroup of a parabolic subgroup $P$ of a reductive
group $\PG(J)$ (not necessarily connected; the definition of a parabolic subgroup
extends appropriately). Moreover, $\PG(J)\cong \Aut(\PG(J)/P)$.
If $J$ is a $27$-dimensional exceptional Jordan algebra, i.e., an Albert algebra,
the group $\PG(J)$ is of type $E_7$ and $P$ is a parabolic subgroup of type $P_7$.
Let $G$ be the corresponding adjoint group of type $E_7$.
Then by~\cite[Th\'er\`eme 1]{Dem77} we have $\Aut(\PG(J)/P)\cong\Aut(G)\cong G$.
Hence $\Aut((J,\,J))$
is isomorphic to a Levi subgroup of a parabolic subgroup $P$ of type $P_7$
in $G$.
\end{proof}

\section{Cubic Jordan algebras and the first Tits construction}

A \emph{cubic map} on a projective $R$-module $V$ consists of a function
$N\colon V\to R$ and its \emph{partial polarization}
$\partial N\colon V\times V\to R$ such that $\partial N(x,\,y)$ is quadratic
in $x$ and linear in $y$, and $N$ is cubic in the following sense:
\begin{itemize}
\item $N(tx)=t^3N(x)$ for all $t\in R$, $x\in V$;
\item $N(x+y)=N(x)+\partial N(x,\,y)+\partial N(y,\,x)+N(y)$ for all $x,\,y\in V$.
\end{itemize}
These data allow to extend $N$ to $V_S=V\otimes_RS$ for any ring extension
$S$ of $R$.

A \emph{cubic Jordan algebra} is a projective module $J$ equipped with a cubic
form $N$, quadratic map $\#\colon J\to J$ and an element $1\in J$ such that
for any extension $S/R$
\begin{itemize}
\item $(x^\#)^\#=N(x)x$ for all $x\in J_S$;
\item $1^\#=1$; $N(1)=1$;
\item $T(x^\#,\,y)=\delta N(x,\,y)$ for all $x,\,y\in J_S$;
\item $1\times x=T(x)1-x$ for all $x\in J_S$,
\end{itemize}
where $\times$ is the linearization of $\#$, $T(x)=\partial N(1,\,x)$,
$T(x,\,y)=T(x)T(y)-N(1,\,x,\,y)$, $N(x,\,y,\,z)$ is the linearization
of $\partial N$.

There is a natural structure of a quadratic Jordan algebra on $J$ given by
the formula
$$
U_xy=T(x,\,y)x-x^\#\times y.
$$

Any associative algebra $A$ of degree $3$ over $R$ (say, commutative \'etale
cubic algebra or an Azumaya algebra of rank $9$) can be naturally considered as
a cubic Jordan algebra, with $N$ being the norm, $T$ being the trace, and
$x^\#$ being the adjoint element to $x$.

Moreover, given an invertible scalar $\lambda\in R^\times$, one can equip
the direct sum $A\oplus A\oplus A$ with the structure of a cubic Jordan
algebra in the following way (which is called the \emph{first Tits construction}):
\begin{align*}
&1=(1,\,0,\,0);\\
&N(a_0,\,a_1,\,a_2)=N(a_0)+\lambda N(a_1)+\lambda^{-1} N(a_2)-T(a_0a_1a_2);\\
&(a_0,\,a_1,\,a_2)^\#=(a_0^\#-a_1a_2,\,\lambda^{-1}a_2^\#-a_0a_1,\,
\lambda a_1^\#-a_2a_0).
\end{align*}

Now we state a transitivity result (borrowed from
\cite[Proof of Theorem~4.8]{PeR}) which is crucial in what follows.

\begin{lem}\label{lem:trans}
Let $E$ be a cubuc \'etale extension of $R$, $A$ is the cubic Jordan algebra
obtained by the first Tits construction from $E$, $y$ be an invertible element
of $E$ considered as a subalgebra of $A$. Then $y$ lies in the orbit of $1$
under the action of subgroup of $\Str(A)(R)$ generated by $\Gm(R)$ and elements
of the form $U_x$, $x$ is an invertible element of $A$.
\end{lem}
\begin{proof}
As an element of $A$ $y$ equals $(y,\,0,\,0)$. Now a direct calculation
shows that
$$
U_{(0,\,0,\,1)}U_{(0,\,y,\,0)}y=N(y)1.
$$
\end{proof}

Over a field, Jordan algebras that can be obtained by the first Tits
construction can be characterized in terms of cohomological invariants.
Namely, to each $J$ one associates a $3$-fold Pfister form $\pi_3(J)$,
and $J$ is of the first Tits construction if and only if $\pi_3(J)$ is
hyperbolic (see \cite[Theorem~4.10]{Pe}). Another equivalent description is
that $J$ splits over a cubic extension of the base field. If the characteristic of the base field
is distinct from $2$, $\pi_3$ is equivalent to the cohomological $f_3$ invariant,
$$
f_3:H^1_{\et}(-,F_4)\to H^3(-,\mu_2).
$$

\section{Springer form}

From now on $J$ is a $27$-dimensional cubic Jordan algebra over $R$.

Let $E$ be a cubic \'etale subalgebra of $J$. Denote by $E^\perp$ the
orthogonal complement to $E$ in $J$ with respect to the bilinear form
$T$ (it exists for the restriction of $T$ to $E$ is non-degenerate); it is a
projective $R$-module of rank $24$. It is shown in \cite[Proposition~2.1]{PeR}
that the operation
\begin{align*}
&E\times E^\perp\to E^\perp;\\
&(a,\,x)\mapsto -a\times x
\end{align*}
equips $E^\perp$ with a structure of $E$-module compatible with its $R$-module structure. Moreover, if we write
$$
x^\#=q_E(x)+r_E(x),\,q_E(x)\in E^\perp,\,r_E(x)\in E,
$$
then $q_E$ is a quadratic form on $E^\perp$, which is nondegenerate as one
can check over a covering of $R$ splitting $J$. This form is called the
\emph{Springer form} with respect to $E$.

The following lemma relates the Springer form and subalgebras of $J$.

\begin{lem}\label{lem:Springer}
Let $v$ be an element of $E^\perp$ such that $q_E(v)=0$ and $v$ is invertible
in $J$. Then $v$ is contained in a subalgebra of $J$ obtained by the first Tits
construction from $E$.
\end{lem}
\begin{proof}
It is shown in \cite[Proof of Proposition~2.2]{PeR} that the embedding
$$
(a_0,\,a_1,\,a_2)\mapsto a_0-a_1\times v-N(v)^{-1}a_2\times v^\#
$$
defines a subalgebra desired.
\end{proof}

Recall that the \'etale algebras of degree $n$ are classified by $\HH^1(R,\,S_n)$,
where $S_n$ is the symmetric group in $n$ letters. The sign map $S_n\to S_2$
induces a map
$$
\HH^1(R,\,S_n)\to\HH^1(R,\,S_2)
$$
that associates to any \'etale algebra $E$ a quadratic \'etale algebra
$\delta(E)$ called the \emph{discriminant} of $E$. The norm $N_{\delta(E)}$
is a quadratic form of rank $2$. We will use later on
the analog of the Grothendieck---Serre conjecture for quadratic \'etale algebras; it follows,
for example, from~\cite[Corollaire 6.1.14]{EGA2}.

Over a field, the Springer form can be computed explicitly in terms of
$\pi_3(J)$ and $\delta(E)$. We will need the following particular case:

\begin{lem}\label{lem:Discr}
Let $J$ be a Jordan algebra over a field $K$ with $\pi_3(J)=0$. Then
$$
q_E={N_{\delta(E)}}_E\perp {\bf h}_E\perp {\bf h}_E\perp {\bf h}_E,
$$
${\bf h}$ stands for the hyperbolic form of rank $2$.
\end{lem}
\begin{proof}
Follows from \cite[Theorem~3.2]{PeR}.
\end{proof}

We will also use the following standard result.

\begin{lem}\label{lem:Conj}
Let $J$ be a Jordan algebra over an algebraically closed field $F$. Then
any two cubic \'etale subalgebras $E$ and $E'$ of $J$ are conjugate by an
element of $\Aut(J)(F)$.
\end{lem}
\begin{proof}
Present $E$ as $Fe_1\oplus Fe_2\oplus Fe_3$, where $e_i$ are idempotents whose
sum is $1$; do the same with $E'$. By \cite[Theorem~17.1]{Lo75} there exists
an element $g\in\Str(J)(F)$ such that $ge_i=e'_i$. But then $g$ stabilizes $1$,
hence belongs to $\Aut(J)(F)$.
\end{proof}

\section{Proof of Theorem~\ref{th:main}}

\begin{proof}[Proof of Theorem~\ref{th:main}]
Set $H=\Aut(J)$. It is a simple group of type $F_4$ over $R$.
We may assume that $H_K$ is not split, otherwise the result follows from
\cite[Theorem~1.0.1]{Pa-newpur}. Let $J$ be the Jordan algebra
corredponding to $H$; we have to show that if $J'$ is a twisted form of $J$
such that $J'_K\simeq J_K$ then $J'\simeq J$. Set $L=\Str(J)$; then $L$ is
a Levi subgroup of a parabolic subgroup of type $P_7$ of an adjoint simple
group scheme $G$ of type $E_7$ by Lemma~\ref{lem:Str}. By~\cite[Exp. XXVI Cor. 5.10 (i)]{SGA} the map
$$
\HH^1_{\et}(R,\,L)\to\HH^1_{\et}(K,\,G)
$$
is injective. Since $G$ is isotropic, by \cite[Theorem~1.0.1]{Pa-newpur}
the map
$$
\HH^1_{\et}(R,\,G)\to\HH^1_{\et}(K,\,G)
$$
has trivial kernel, and so does the map
$$
\HH^1_{\et}(R,\,L)\to\HH^1_{\et}(K,\,L).
$$

But $(J'_K,\,J'_K)\simeq (J_K,\,J_K)$, therefore $(J',\,J')\simeq (J,\,J)$,
that is $J'$ is isomorphic to $J^{(y)}$ for some invertible $y\in J$. It
remains to show that $y$ lies in the orbit of $1$ under the action of
$\Str(J)(R)$.

Present the quotient of $R$ by its Jacobson radical as a direct product of
the residue fields $\prod k_i$. An argument in \cite[Proof of Theorem~4.8]{PeR}
shows that for each $i$ one can find an invertible element $v_i\in J_{k_i}$ such
that the discriminant of the generic polynomial of $U_{v_i}y_{k_i}$ is nonzero.
Lifting $v_i$ to an element $v\in J$ and changing $y$ to $U_vy$ we may assume
that the generic polynomial $f(T)\in R[T]$ of $y$ has the property that
$R[T]/(f(T))$ is an \'etale extension of $R$. In other words, we may assume
that $y$ generates a cubic \'etale subalgebra $E$ in $J$.

Note that $E_K$ is a cubic field extension of K; otherwise $J_K$ is reduced,
hence split, for $\pi_3(J_K)=0$ (see \cite[Theorem~4.10]{Pe}). Consider the
form
$$
q={N_{\delta(E)}}_E\perp{\bf h}_E\perp{\bf h}_E\perp{\bf h}_E;
$$
then by Lemma~\ref{lem:Discr} $q_K=q_{E_K}$. By the analog of the
Grothendieck---Serre conjecture for \'etale quadratic algebras, $q$ and $q_E$ have the same discriminant. So $q_E$ is a twisted
form of $q$ given by a cocycle $\xi\in\HH^1(E,\,\SO(q))$. Now $\xi_K$ is trivial,
and \cite[Theorem~1.0.1]{Pa-newpur} imply that $\xi$ is trivial
itself, that is $q_E=q$. In particular, $q_E$  is isotropic. Let us show that
there is an \emph{invertible} element $v$ in $J$ such that $q_E(v)=0$.

The projective quadric over $E$ defined by $q_E$ is isotropic, hence has an open
subscheme $U\simeq{\mathbb A}_E^n$. Denote by $U'$ the open subscheme of $R_{E/R}(U)$
consisting of invertible elements. It suffices to show that $U'(k_i)$ is
non-empty for each $i$, or, since the condition on $R$ implies that $k_i$
is infinite, that $U'(\bar k_i)$ is non-empty.

But $J_{\bar k_i}$ splits, and, in particular, it is obtained by a first Tits construction
from a split Jordan algebra of $3\times 3$ matrices over $\bar k_i$. The diagonal
matrices in this matrix algebra constitute a cubic \'etale subalgebra of $J_{\bar k_i}$.
By Lemma~\ref{lem:Conj} we may assume that this \'etale subalgebra coincides
with $E_{\bar k_i}$. By~\cite[Proposition~2.2]{PeR} there exists an invertible
element $v_i\in E^{\perp}_{\bar k_i}$ such that $q_{E_{\bar k_i}}(v_i)=0$. Thus
the scheme of invertile elements intersects the quadric over $\bar k_i$, hence,
$U'(\bar k_i)$ is non-empty.

Finally, Lemma~\ref{lem:Springer} and Lemma~\ref{lem:trans} show that $y$
belongs to the orbit of $1$ under the group generated by $\Gm(R)$ and
elements of the form $U_x$. So $J'\simeq J^{(y)}\simeq J$, and the proof is
completed.
\end{proof}

The authors are heartily grateful to Ivan Panin, who introduced them to the subject and
provided inspiring comments during the course of the work.

\end{document}